\documentclass[12pt,a4paper]{amsart}
\usepackage{amsmath}
\usepackage{amsthm}
\usepackage{amssymb}
\usepackage[english]{babel}
\usepackage{verbatim}
\usepackage[shortlabels]{enumitem}
\usepackage{color}
\usepackage{tikz-cd}

\newcommand{\insfracid}{\mathcal{F}}
\newcommand{\insfracidfg}{\mathcal{F}_f}
\newcommand{\inssubmod}{\mathbf{F}}
\newcommand{\Zarmin}{\mathrm{Zar_{min}}}
\newcommand{\NoethOver}{\mathrm{NoethOver}}
\newcommand{\KrullOver}{\mathrm{KrullOver}}
\newcommand{\RegOver}{\mathrm{RegOver}}
\newcommand{\Kr}{\mathrm{Kr}}
\newcommand{\Krset}{\mathcal{K}}
\newcommand{\gen}{\mathrm{gen}}
\newcommand{\cont}{\mathbf{c}}
\newcommand{\Cl}{\mathrm{Cl}}

\DeclareMathOperator{\trdeg}{trdeg}
\DeclareMathOperator{\Spec}{Spec}
\DeclareMathOperator{\Max}{Max}
\DeclareMathOperator{\rad}{rad}
\DeclareMathOperator{\Zar}{Zar}
\DeclareMathOperator{\Over}{Over}

\newcommand{\ins}[1]{\mathbb{#1}}
\newcommand{\insN}{\ins{N}}
\newcommand{\insZ}{\ins{Z}}

\newcommand{\inN}{\in\insN}

\newtheoremstyle{mio}%
	{}{} 
	{\itshape}{} 
	{\bfseries}{.}{ } 
	{#1 #2\thmnote{~\mdseries(#3)}} 
\theoremstyle{mio}
\newtheorem{teor}{Theorem}[section]
\newtheorem{cor}[teor]{Corollary}
\newtheorem{prop}[teor]{Proposition}
\newtheorem{lemma}[teor]{Lemma}
\newtheorem{defin}[teor]{Definition}

\newtheoremstyle{definition2}%
	{}{} 
	{}{} 
	{\bfseries}{.}{ } 
	{#1 #2\thmnote{\mdseries~ #3}} 
\theoremstyle{definition2}
\newtheorem{ex}[teor]{Example}
\newtheorem{oss}[teor]{Remark}

\title[Non-compact subsets of the Zariski space]{Non-compact subsets of the Zariski space of an integral domain}
\author{Dario Spirito}
\email{spirito@mat.uniroma3.it}
\address{Dipartimento di Matematica e Fisica, Universit\`a degli Studi
``Roma Tre'', Roma, Italy}
\date{\today}
\thanks{This work was partially supported by {\sl GNSAGA} of {\sl Istituto Nazionale di Alta Matematica}.}

\begin{document}

\begin{abstract}
Let $V$ be a minimal valuation overring of an integral domain $D$ and let $\Zar(D)$ be the Zariski space of the valuation overrings of $D$. Starting from a result in the theory of semistar operations, we prove a criterion under which the set $\Zar(D)\setminus\{V\}$ is not compact. We then use it to prove that, in many cases, $\Zar(D)$ is not a Noetherian space, and apply it to the study of the spaces of Kronecker function rings and of Noetherian overrings.
\end{abstract}

\maketitle

\section{Introduction}
The \emph{Zariski space} $\Zar(K|D)$ of the valuation rings of a field $K$ containing a domain $D$ was introduced (under the name \emph{abstract Riemann surface}) by O. Zariski, who used it to show that resolution of singularities holds for varieties of dimension 2 or 3 over fields of characteristic 0 \cite{zariski_sing,zariski_comp}. In particular, Zariski showed that $\Zar(K|D)$, endowed with a natural topology, is always a compact space \cite[Chapter VI, Theorem 40]{zariski_samuel_II}; this result has been subsequently improved by showing that $\Zar(K|D)$ is a spectral space (in the sense of Hochster \cite{hochster_spectral}), first in the case where $K$ is the quotient field of $D$ \cite{dobbs_fedder_fontana,fontana_krr-abRs}, and then in the general case \cite[Corollary 3.6(3)]{fifolo_transactions}. The topological aspects of the Zariski space has subsequently been used, for example, in real and rigid algebraic geometry \cite{hub-kneb,schwartz-compactification} and in the study of representation of integral domains as intersections of valuation overrings \cite{olberding_noetherianspaces,olberding_affineschemes,olberding_topasp}. In the latter context, i.e., when $K$ is the quotient field of $D$, two important properties for subspaces of $\Zar(K|D)$ to investigate are the properties of compactness and of Noetherianess.

In this paper, we concentrate on the case where $K$ is the quotient field of $D$, studying subspaces of $\Zar(K|D)=\Zar(D)$ that are \emph{not} compact. The starting point is a criterion based on semistar operations, proved in \cite[Theorems 4.9 and 4.13]{fifolo_transactions} (see also \cite[Proposition 4.5]{topological-cons} for a slightly stronger version) and integrated, as in \cite[Example 3.7]{surveygraz}, with the use of the two-faced definition of the integral closure/$b$-operation, either through valuation overrings or through equations of integral dependence (see e.g. \cite[Chapter 6]{swanson_huneke}). In particular, we  analyze sets of the form $\Zar(D)\setminus\{V\}$, where $V$ is a minimal valuation overring of $D$: we show in Section \ref{sect:Vmin} that such a space is compact only if $V$ can be obtained from $D$ in a very specific way (more precisely, as the integral closure of a localization of a finitely generated algebra over $D$), and we follow up in Sections \ref{sect:dimV} and \ref{sect:intersection} by showing that this condition implies a bound on the dimension of $V$ in relation with the dimension of $D$ (Proposition \ref{prop:leq2dim}) and a quite strict condition on the intersection of sets of prime ideals of $D$ (Theorem \ref{teor:intersez-P}). Section \ref{sect:kronecker} is dedicated to a brief application of these criteria to the study of Kronecker function rings (the definition will be recalled later).

In Section \ref{sect:Noeth}, we consider the set $\Over(D)$ of overrings of $D$ (which is known to be itself a spectral space \cite[Proposition 3.5]{finocchiaro-ultrafiltri}). Using the result proved in the previous sections, we show that, when $D$ is a Noetherian domain, some distinguished subspaces of $\Over(D)$ (for example, the subspace of overrings of $D$ that are Noetherian) are not spectral.

\section{Preliminaries and notation}
\subsection{Spectral spaces}
A topological space $X$ is a \emph{spectral space} if there is a ring $R$ such that $X$ is homeomorphic to the prime spectrum $\Spec(R)$, endowed with the Zariski topology. Spectral spaces can be characterized in a purely topological way as those spaces that are $T_0$, compact, with a basis of open and compact subset that is closed by finite intersections and such that every irreducible closed subset has a generic point (i.e., it is the closure of a single point) \cite[Proposition 4]{hochster_spectral}.

On a spectral space $X$ it is possible to define two new topologies: the \emph{inverse} and the \emph{constructible} topology.

The \emph{inverse topology} is the topology on $X$ having, as a basis of closed sets, the family of open and compact subspaces of $X$. Endowed with the inverse topology, $X$ is again a spectral space \cite[Proposition 8]{hochster_spectral}; moreover, a subspace $Y\subseteq X$ is closed in the inverse topology if and only if $Y$ is compact (in the original topology) and $Y=Y^{\gen}$ \cite[Remark 2.2 and Proposition 2.6]{fifolo_transactions}, where
\begin{equation*}
\begin{array}{rcl}
Y^\gen & := & \{z\in X\mid z\leq y\text{~for some~}y\in Y\}=\\
& = & \{z\in X\mid y\in\Cl(z)\text{~for some~}y\in Y\},
\end{array}
\end{equation*}
with $\Cl(z)$ denoting the closure of the singleton $\{z\}$ (again, in the original topology) and $\leq$ is the order induced by the original topology \cite[d-1]{encytop}, which coincides on $\Spec(R)$ with the set-theoretic inclusion. 

The \emph{constructible topology} on $X$ (also called \emph{patch topology}) is the coarsest topology such that the open and compact subsets of $X$ are both open and closed. Endowed with the constructible topology, $X$ is a spectral space that is also Haussdorff (see \cite[Propositions 3 and 5]{olivier_plats_1}, \cite{olivier_plats_2} or \cite[Proposition 5]{fontana_patch}), and the constructible topology is finer than both the original and the inverse topology. A subset of $X$ closed in the constructible topology is said to be a \emph{proconstructible subset} of $X$; if $Y$ is proconstructible, then it is a spectral space when endowed with the topology induced by the original spectral topology of $X$, and the constructible topology on $Y$ is exactly the topology induced by the constructible topology on $X$ (this follows from \cite[1.9.5(vi-vii)]{EGA4-1}).

\subsection{Noetherian spaces}
A topological space $X$ is \emph{Noetherian} if $X$ verifies the ascending chain condition on the open subsets, or equivalently if every subspace of $X$ is compact. Examples of Noetherian spaces are finite spaces and the prime spectra of Noetherian rings. If $\Spec(R)$ is a Noetherian space, then every proper ideal of $R$ has only finitely many minimal primes (see e.g. the proof of \cite[Chapter 4, Corollary 3, p.102]{bourbaki_ac} or \cite[Chapter 6, Exercises 5 and 7]{atiyah}).

\subsection{Overrings and the Zariski space}
Let $D\subseteq K$ be an extension of integral domains. We denote the set of all rings contained between $D$ and $K$ by $\Over(K|D)$; if $K$ is a field (not necessarily the quotient field of $D$), the set of all valuation rings containing $D$ with quotient field $K$ is denoted by $\Zar(K|D)$, and it is called the \emph{Zariski space} (or the \emph{Zariski-Riemann space}) of $D$.

The \emph{Zariski topology} on $\Over(K|D)$ is the topology having, as a subbasis, the sets of the form
\begin{equation*}
B(x_1,\ldots,x_n):=\{T\in\Over(K|D)\mid x_1,\ldots,x_n\in T\},
\end{equation*}
as $\{x_1,\ldots,x_n\}$ ranges among the finite subsets of $K$. Under this topology, both $\Over(K|D)$ \cite[Proposition 3.5]{finocchiaro-ultrafiltri} and its subspace $\Zar(K|D)$ \cite{fontana_krr-abRs,dobbs_fedder_fontana} are spectral spaces, and the order induced by this topology is the inverse of the set-theoretic inclusion. In particular, every $Y\subseteq\Over(K|D)$ with a minimum element is compact, and, if $Z$ is an arbitrary subset of $\Over(K|D)$, then $Z^{\gen}=\{T\in\Over(K|D)\mid T\supseteq A\text{~for some~}A\in Z\}$.

We denote by $\Zarmin(D)$ the set of minimal elements of $\Zar(D)$; since $\Zar(D)$ is a spectral space, every $V\in\Zar(D)$ contains an element $W\in\Zarmin(D)$.

If $K$ is the quotient field of $D$, then we set $\Over(K|D)=:\Over(D)$ and $\Zar(K|D)=:\Zar(D)$. Elements of $\Over(D)$ are called \emph{overrings} of $D$, elements of $\Zar(D)$ are the \emph{valuation overrings} of $D$ and elements of $\Zarmin(D)$ are the \emph{minimal valuation overrings} of $D$.

The \emph{center map} is the application
\begin{equation*}
\begin{aligned}
\gamma\colon\Zar(K|D) & \longrightarrow \Spec(D)\\
V & \longmapsto \mathfrak{m}_V\cap D,
\end{aligned}
\end{equation*}
where $\mathfrak{m}_V$ is the maximal ideal of $V$. When $\Zar(K|D)$ and $\Spec(D)$ are endowed with the respective Zariski topologies, the map $\gamma$ is continuous (\cite[Chapter VI, \textsection 17, Lemma 1]{zariski_samuel_II} or \cite[Lemma 2.1]{dobbs_fedder_fontana}), surjective (this follows, for example, from \cite[Theorem 5.21]{atiyah} or \cite[Theorem 19.6]{gilmer}) and closed \cite[Theorem 2.5]{dobbs_fedder_fontana}.

\subsection{Semistar operations}\label{sect:semistar}
Let $D$ be a domain with quotient field $K$. Let $\inssubmod(D)$ be the set of $D$-submodules of $K$, $\insfracid(D)$ be the set of fractional ideals of $D$, and $\insfracidfg(D)$ be the set of finitely generated fractional ideals of $D$.

A \emph{semistar operation} on $D$ is a map $\star:\inssubmod(D)\longrightarrow\inssubmod(D)$, $I\mapsto I^\star$, such that, for every $I,J\in\inssubmod(D)$ and every $x\in K$,
\begin{enumerate}
\item $I\subseteq I^\star$;
\item if $I\subseteq J$, then $I^\star\subseteq J^\star$;
\item $(I^\star)^\star=I^\star$;
\item $x\cdot I^\star=(xI)^\star$.
\end{enumerate}
Given a semistar operation $\star$, the map $\star_f$ is defined on every $E\in\inssubmod(D)$ by
\begin{equation*}
E^{\star_f}=\bigcup\{F^\star\mid F\in\insfracidfg(D), F\subseteq E\}.
\end{equation*}
The map $\star_f$ is always a semistar operation; if $\star=\star_f$, then $\star$ is said to be \emph{of finite type}. Two semistar operations of finite type $\star_1,\star_2$ are equal if and only if $I^{\star_1}=I^{\star_2}$ for every $I\in\insfracidfg(D)$. See \cite{okabe-matsuda} for general informations about semistar operations.

If $\Delta\subseteq\Zar(D)$, then $\wedge_\Delta$ is defined as the semistar operation on $D$ such that
\begin{equation*}
I^{\wedge_\Delta}:=\bigcap\{IV\mid V\in\Delta\}
\end{equation*}
for every $D$-submodule $I$ of $K$; a semistar operation of type $\wedge_\Delta$ is said to be a \emph{valuative semistar operation}. By \cite[Proposition 4.5]{topological-cons}, $\wedge_\Delta$ is of finite type if and only if $\Delta$ is compact (in the Zariski topology of $\Zar(D)$). If $\Delta,\Lambda\subseteq\Zar(D)$, then $\wedge_\Delta=\wedge_\Lambda$ if and only if $\Delta^\gen=\Lambda^\gen$ \cite[Lemma 5.8(1)]{spettrali-eab}, while $(\wedge_\Delta)_f=(\wedge_\Lambda)_f$ if and only if $\Delta$ and $\Lambda$ have the same closure with respect to the inverse topology \cite[Theorem 4.9]{fifolo_transactions}. The semistar operation $\wedge_{\Zar(D)}$ is usually denoted by $b$ and called the \emph{$b$-operation}.

\section{The use of minimal valuation domains}\label{sect:Vmin}
The starting point of this paper is the following well-known result.
\begin{prop}[see e.g. {{\protect\cite[Proposition 6.8.2]{swanson_huneke}}}]\label{prop:ic-b}
Let $I$ be an ideal of an integral domain $D$; let $x\in D$. Then, $x\in IV$ for every $V\in\Zar(D)$ if and only if there are $n\geq 1$ and $a_1,\ldots,a_n\in D$ such that $a_i\in I^i$ and 
\begin{equation}\label{eq:ic}
x^n+a_1x^{n-1}+\cdots+a_{n-1}x+a_n=0.
\end{equation}
\end{prop}

An inspection of the proof of the previous proposition given in \cite{swanson_huneke} shows that this result does not really rely on the fact that $I$ is an ideal of $D$, or on the fact that $x\in D$; indeed, it applies to every $D$-submodule $I$ of the quotient field $K$, and to every $x\in K$. In the terminology of semistar operations, this means that, for each $I\in\inssubmod(D)$, $I^b=I^{\wedge_{\Zar(D)}}$ is exactly the set of $x\in K$ that verifies an equation like \eqref{eq:ic}, with $a_i\in I^i$. We are interested in generalizing that proof in a different way; we need the following definitions.
\begin{defin}
Let $D$ be an integral domain and let $\Delta,\Lambda\subseteq\Over(D)$. We say that $\Lambda$ \emph{dominates} $\Delta$ if, for every $T\in\Delta$ and every $M\in\Max(T)$, there is a $A\in\Lambda$ such that $T\subseteq A$ and $1\notin MA$.
\end{defin}
For example, $\Zar(D)$ dominates every subset of $\Over(D)$, while the set of localizations of $D$ dominates $\{D\}$.

\begin{defin}
Let $D$ be an integral domain domain. We denote by $D[\insfracidfg]$ the set of finitely generated $D$-algebras of $\Over(D)$, or equivalently
\begin{equation*}
D[\insfracidfg]:=\{D[I]:I\in\insfracidfg(R)\}.
\end{equation*}
\end{defin}

Even if the proof of the following result essentially repeats the proof of \cite[Proposition 6.8.2]{swanson_huneke}, we replay it here for clarity.
\begin{prop}\label{prop:domina-Ff-I}
Let $D$ be an integral domain, and suppose that $\Delta\subseteq\Zar(D)$ dominates $D[\insfracidfg]$. Then, for every finitely generated ideal $I$ of $D$, $I^{\wedge_\Delta}=I^b$.
\end{prop}
\begin{proof}
Clearly, $I^b\subseteq I^{\wedge_\Delta}$. Suppose thus that $x\in I^{\wedge_\Delta}$, $x\neq 0$, and let $I=(i_1,\ldots,i_k)D$. Define $J:=x^{-1}I\in\insfracidfg(D)$, and let $A:=D[J]=D[x^{-1}i_1,\ldots,x^{-1}i_k]$; by definition, $J\subseteq A$.

If $JA\neq A$, then there is a maximal ideal $M$ of $A$ containing $J$, and thus, by domination, there is a valuation domain $V\in\Delta$ containing $A$ whose maximal ideal $\mathfrak{m}_V$ is such that $JV\subseteq\mathfrak{m}_V$, and thus $IV\subseteq x\mathfrak{m}_V$. However, $x\in I^b\subseteq IV$, which implies $x\in x\mathfrak{m}_V$, a contradiction.

Hence, $JA=A$, i.e., $1=j_1a_1+\cdots+j_na_n$ for some $j_t\in J$, $a_t\in A$; expliciting the elements of $A$ as elements of $D[J]$ and using $J=x^{-1}I$, we find that there must be an $N\inN$ and elements $i_t\in I^t$ such that $x^N=i_1x^{N-1}+\cdots+i_{N-1}x+i_N$, which gives an equation of integral dependence of $x$ over $I$. Therefore, $x\in I^b$, as requested.
\end{proof}

We can now use the properties of valuative semistar operations to study compactness.
\begin{prop}\label{prop:comp-Zarmin}
Let $D$ be an integral domain, and let $\Delta\subseteq\Zar(D)$ be a set that dominates $D[\insfracidfg]$. Then, $\Delta$ is compact if and only if it contains $\Zarmin(D)$.
\end{prop}
\begin{proof}
If $\Delta$ contains $\Zarmin(D)$, then $\mathcal{U}$ is an open cover of $\Delta$ if and only if it is an open cover of $\Zar(D)$; thus, $\Delta$ is compact since $\Zar(D)$ is.

Conversely, suppose $\Delta$ is compact. By Proposition \ref{prop:domina-Ff-I}, $I^{\wedge_\Delta}=I^b$ for every finitely generated ideal $I$; hence, $(\wedge_\Delta)_f=b_f=b$. By \cite[Lemma 5.8(1)]{spettrali-eab}, it follows that the closure of $\Delta$ with respect to the inverse topology of $\Zar(D)$ is the whole $\Zar(D)$; however, since $\Delta$ is compact, its closure in the inverse topology is exactly $\Delta^\gen=\Delta^\uparrow=\{W\in\Zar(D)\mid W\supseteq V\text{~for some~}V\in\Delta\}$. Hence, $\Delta$ must contain $\Zarmin(D)$.
\end{proof}

Thus, to find a subset of $\Zar(D)$ that is not compact, it is enough to find a $\Delta$ that dominates $D[\insfracidfg]$ but that does not contain $\Zarmin(D)$. The easiest case where this criterion can be applied is when $\Delta=\Zar(D)\setminus\{V\}$ for some $V\in\Zarmin(D)$.
\begin{teor}\label{teor:Zar-meno-V}
Let $D$ be an integral domain and let $V\in\Zarmin(D)$. If $\Zar(D)\setminus\{V\}$ is compact, then $V$ is the integral closure of $D[x_1,\ldots,x_n]_M$ for some $x_1,\ldots,x_n\in K$ and some $M\in\Max(D[x_1,\ldots,x_n])$.
\end{teor}
\begin{proof}
If $\Delta:=\Zar(D)\setminus\{V\}$ is compact, then by Proposition \ref{prop:comp-Zarmin} it cannot dominate $D[\insfracidfg]$. Hence, there is a finitely generated fractional ideal $I$ such that $\Delta$ does not dominate $A:=D[I]$, and so a maximal ideal $M$ of $A$ such that $1\in MW$ for every $W\in\Delta$. In particular, $A\neq K$ (otherwise $M$ would be $(0)$).

However, there must be a valuation ring containing $A_M$ whose center (on $A_M$) is $MA_M$, and the unique possibility for this valuation ring is $V$: it follows that $V$ is the unique valuation ring centered on $MA_M$. However, the integral closure of $A_M$ is the intersection of the valuation rings with center $MA_M$ (since every valuation ring containing $A_M$ contains a valuation ring centered on $MA_M$ \cite[Corollary 19.7]{gilmer}); thus, $V$ is the integral closure of $A_M$.
\end{proof}

\section{The dimension of $V$}\label{sect:dimV}
Before embarking on using Theorem \ref{teor:Zar-meno-V}, we prove a simple yet general result.
\begin{prop}\label{prop:spec-noeth}
Let $D$ be an integral domain. If $\Zar(D)$ is a Noetherian space, so is $\Spec(D)$.
\end{prop}
\begin{proof}
The claim follows from the fact that $\Spec(D)$ is the continuous image of $\Zar(D)$ through the center map $\gamma$, and that the image of a Noetherian space is still Noetherian.
\end{proof}

Note that the converse of this proposition is far from being true (this is, for example, a consequence of Proposition \ref{prop:polynomial} or of Proposition \ref{prop:noethdom}).

The problem in using Theorem \ref{teor:Zar-meno-V} is that it is usually difficult to control the behaviour of finitely generated algebras over $D$. We can, however, control the behaviour of the prime spectrum of $D$.
\begin{lemma}\label{lemma:spec-linord}
Let $D$ be an integral domain, and let $V\in\Zar(D)$ be the integral closure of  $D_M$, for some $M\in\Spec(D)$. Then, the set of prime ideals of $D$ contained in $M$ is linearly ordered.
\end{lemma}
\begin{proof}
Let $P,Q$ be two prime ideals of $D$ contained in $M$; then, $PD_M,QD_M\in\Spec(D_M)$. Since $D_M\subseteq V$ is an integral extension, $PD_M=P'\cap D_M$ and $QD_M=Q'\cap D_M$ for some $P',Q'\in\Spec(V)$; however, $V$ is a valuation domain, and thus (without loss of generality) $P'\subseteq Q'$. Hence, $PD_M\subseteq QD_M$ and $P\subseteq Q$, as requested.
\end{proof}

\begin{prop}\label{prop:leq2dim}
Let $D$ be an integral domain, let $V\in\Zarmin(D)$ and suppose that $\Zar(D)\setminus\{V\}$ is compact. Let $\iota_V:\Spec(V)\longrightarrow\Spec(D)$ be the canonical spectral map associated to the inclusion $D\hookrightarrow V$. For every $P\in\Spec(D)$, $|\iota_V^{-1}(P)|\leq 2$; in particular, $\dim(V)\leq 2\dim(D)$.
\end{prop}
\begin{proof}
Suppose $|\iota_V^{-1}(P)|>2$: then, there are prime ideals $Q_1\subsetneq Q_2\subsetneq Q_3$ of $V$ such that $\iota_V(Q_1)=\iota_V(Q_2)=\iota_V(Q_3)=:P$. If $\Zar(D)\setminus\{V\}$ is compact, by Theorem \ref{teor:Zar-meno-V} there is a finitely generated $D$-algebra $A:=D[a_1,\ldots,a_n]$ such that $V$ is the integral closure of $A_M$, for some maximal ideal $M$ of $A$. We can write $A_M$ as a quotient $\frac{D[X_1,\ldots,X_n]_\mathfrak{a}}{\mathfrak{b}}$, where $X_1,\ldots,X_n$ are independent indeterminates and $\mathfrak{a},\mathfrak{b}\in\Spec(D[X_1,\ldots,X_n])$. Since $A_M\subseteq V$ is an integral extension, $Q_i\cap A\neq Q_j\cap A$ if $i\neq j$.

For $i\in\{1,2,3\}$, let $\mathfrak{q}_i$ be the prime ideal of $D[X_1,\ldots,X_n]$ whose image in $A$ is $Q_i$; then, $\mathfrak{q}_1$, $\mathfrak{q}_2$ and $\mathfrak{q}_3$ are distinct, $\mathfrak{q}_i\cap D=P$ for each $i$, and the set of ideals between $\mathfrak{q}_1$ and $\mathfrak{q}_3$ is linearly ordered (by Lemma \ref{lemma:spec-linord}). However, the prime ideals of $D[X_1,\ldots,X_n]$ contracting to $P$ are in a bijective and order-preserving correspondence with the prime ideals of $F[X_1,\ldots,X_n]$, where $F$ is the quotient field of $D/P$; since $F[X_1,\ldots,X_n]$ is a Noetherian ring, there are an infinite number of prime ideals between the ideals corresponding to $\mathfrak{q}_1$ and $\mathfrak{q}_3$. This is a contradiction, and $|\iota_V^{-1}(P)|\leq 2$.

For the ``in particular'' statement, take a chain $(0)\subsetneq Q_1\subsetneq\cdots\subsetneq Q_k$ in $\Spec(V)$. Then, the corresponding chain of the $P_i:=Q_i\cap D$ has length at most $\dim(D)$, and moreover $\iota^{-1}((0))=\{(0)\}$. Hence, $k+1\leq 2\dim(D)+1$ and $\dim(V)\leq 2\dim(D)$.
\begin{figure}\label{fig:prop:leq2dim}
\begin{equation*}
\begin{tikzcd}
& D[X_1,\ldots,X_n]\arrow[two heads]{d}\arrow[hook]{r} & D[X_1,\ldots,X_n]_\mathfrak{a}\arrow[two heads]{d}\\
D\arrow[end anchor=south west,hook]{ur}\arrow[hook]{r} & A=D[a_1,\ldots,a_n]\arrow[hook]{r} & A_M\simeq\frac{D[X_1,\ldots,X_n]_\mathfrak{a}}{\mathfrak{b}}\arrow[hook]{r} & V
\end{tikzcd}
\end{equation*}
\caption{Rings involved in the proof of Proposition \ref{prop:leq2dim}.}
\end{figure}
\end{proof}

The \emph{valuative dimension} of $D$, indicated by $\dim_v(D)$, is defined as the supremum of the dimensions of the valuation overrings of $D$; we have always $\dim(D)\leq\dim_v(D)$, and $\dim_v(D)$ can be arbitrarily large with respect to $\dim(D)$ \cite[Section 30, Exercises 16 and 17]{gilmer}. In particular, with the notation of the previous proposition, the cardinality of $\iota_V^{-1}(P)$ can be arbitrarily large: for example, if $(D,\mathfrak{m})$ is local and one-dimensional, then $|\iota_V^{-1}(\mathfrak{m})|=\dim_v(D)$. 
\begin{cor}\label{cor:dimv-dim}
Let $D$ be an integral domain such that $\Zar(D)$ is Noetherian. Then, $\dim_v(D)\leq 2\dim(D)$.
\end{cor}
\begin{proof}
If $\Zar(D)$ is Noetherian, then in particular $\Zar(D)\setminus\{V\}$ is compact for every $V\in\Zarmin(D)$. Hence, $\dim(V)\leq 2\dim(D)$ for every $V\in\Zarmin(D)$, by Proposition \ref{prop:leq2dim}; since, if $W\supseteq V$ are valuation domain, $\dim(W)\leq\dim(V)$, the claim follows.
\end{proof}

\begin{prop}\label{prop:leq2dim-meglio}
Let $D$ be an integral domain, and let $V\in\Zarmin(D)$ be such that $\Zar(D)\setminus\{V\}$ is compact; let $(0)\subsetneq P_1\subsetneq\cdots\subsetneq P_k$ be the chain of prime ideals of $V$ and let $Q_i:=P_i\cap D$. Denote by $ht(P)$ the height of the prime ideal $P$. Then:
\begin{enumerate}[(a)]
\item\label{prop:leq2dim-meglio:alt} for every $0\leq t\leq\dim(D)$, we have
\begin{equation*}
\dim(V)\leq\dim_v(D_{Q_t})+2(\dim(D)-ht(Q_t));
\end{equation*}
\item\label{prop:leq2dim-meglio:val} if $D_{Q_t}$ is a valuation domain, then
\begin{equation*}
\dim(V)\leq 2\dim(D)-ht(Q_t).
\end{equation*}
\end{enumerate}
\end{prop}
\begin{proof}
\ref{prop:leq2dim-meglio:alt} Let $(0)\subsetneq Q^{(1)}\subsetneq Q^{(2)}\subsetneq\cdots\subsetneq Q^{(s)}$ be the chain $(0)\subseteq Q_1\subseteq\cdots\subseteq Q_k$ without the repetitions, and let $a$ be the index such that $Q^{(a)}=Q_t$. For every $b>a$, by the proof of Proposition \ref{prop:leq2dim} there can be at most two prime ideals of $V$ over $Q^{(b)}$; on the other hand, $V_{P_t}$ is a valuation overring of $D_{Q_t}$, and thus $t=\dim(V_{P_t})\leq\dim_v(D_{Q_t})$. Therefore,
\begin{equation*}
\dim(V)\leq t+2(s-a)\leq\dim_v(D_{Q_t})+2(\dim(D)-ht(Q_t))
\end{equation*}
since each ascending chain of prime ideals starting from $Q_t$ has length at most $\dim(D)-ht(Q_t)$.

Point \ref{prop:leq2dim-meglio:val} follows, since $\dim(V)=\dim_v(V)$ for every valuation domain $V$.
\end{proof}

\begin{ex}\label{ex:prufnoeth}
A class of integral domain whose Zariski space is Noetherian is constituted by the class of Pr\"ufer domains with Noetherian spectrum. Indeed, if $D$ is a Pr\"ufer domain then the valuation overrings of $D$ are exactly the localizations of $D$ at prime ideals; thus, the center map $\gamma$ estabilishes a homeomorphism between $\Zar(D)$ and $\Spec(D)$. Thus, if the latter is Noetherian also the former is Noetherian.

In this case, $\dim(D)=\dim_v(D)$.
\end{ex}

\begin{ex}\label{ex:dim12}
It is also possible to construct domains whose Zariski space is Noetherian but with $\dim(D)\neq\dim_v(D)$. For example, let $L$ be a field, and consider the ring $A:=L+YL(X)[[Y]]$, where $X$ and $Y$ are independent indeterminates. Then, the valuation overrings of $A$ different from $F:=L(X)((Y))$ are the rings in the form $V+YL(X)[[Y]]$, as $V$ ranges among the valuation rings containing $L$ and having quotient field $L(X)$; that is, $\Zar(A)\setminus\{F\}\simeq\Zar(L(X)|L)$. By the following Corollary \ref{cor:FL}, $\Zar(A)$ is a Noetherian space.

From this, we can construct analogous examples of arbitrarily large dimension. Indeed, if $R$ is an integral domain with quotient field $K$, and $T:=R+XK[[X]]$, then as above $\Zar(T)$ is composed by $K((X))$ and by rings of the form $V+XK[[X]]$, as $V$ ranges in $\Zar(R)$; in particular, $\Zar(T)=\{K((X))\}\cup\mathcal{X}$, where $\mathcal{X}\simeq\Zar(R)$. Thus, $\Zar(T)$ is Noetherian if $\Zar(R)$ is. Moreover, $\dim(T)=\dim(R)+1$ and $\dim_v(T)=\dim_v(R)+1$.

Consider now the sequence of rings $R_1:=L+YL(X)[[Y]]$, $R_2:=R_1+Y_2Q(R_1)[[Y_2]]$, \ldots, $R_n:=R_{n-1}+Y_nQ(R_{n-1})[[Y_n]]$, where $Q(R)$ indicates the quotient field of $R$ and each $Y_i$ is an indeterminate over $Q(R_{i-1})((Y_{i-1}))$. Recursively, we see that each $\Zar(R_n)$ is Noetherian, while $\dim(R_n)=n\neq n+1=\dim_v(R_n)$.
\end{ex}

\section{Intersections of prime ideals}\label{sect:intersection}
The results of the previous sections, while very general, are often difficult to apply, because it is usually not easy to determine the valuative dimension of a domain $D$. More applicable criteria, based on the prime spectrum of $D$, are the ones that we will prove next.
\begin{teor}\label{teor:intersez-P}
Let $D$ be a local integral domain, and suppose there is a set $\Delta\subseteq\Spec(D)$ and a prime ideal $Q$ such that:
\begin{enumerate}
\item $Q\notin\Delta$;
\item no two members of $\Delta$ are comparable;
\item $\bigcap\{P\mid P\in\Delta\}=Q$;
\item $D_Q$ is a valuation domain.
\end{enumerate}
Then, for any minimal valuation overring $V$ of $D$ contained in $D_Q$, $\Zar(D)\setminus\{V\}$ is not compact; in particular, $\Zar(D)$ is not Noetherian.
\end{teor}
\begin{proof}
Note first that, since $V$ is a minimal valuation overring, its center $M$ on $D$ must be the maximal ideal of $D$ \cite[Corollary 19.7]{gilmer}. Suppose that $\Zar(D)\setminus\{V\}$ is compact: by Theorem \ref{teor:Zar-meno-V}, there is a finitely generated $D$-algebra $A:=D[x_1,\ldots,x_n]$ such that $V$ is the integral closure of $A_M$ for some $M\in\Max(A)$. 

Let $I:=x_1^{-1}D\cap\cdots\cap x_n^{-1}D\cap D=(D:_Dx_1)\cap\cdots\cap(D:_Dx_n)$. If $I\subseteq Q$, then $(D:_Dx)\subseteq Q$ for some $x_i:=x$; then, since $D_Q$ is flat over $D$,
\begin{equation*}
(D_Q:_{D_Q}x)=(D:_Dx)D_Q\subseteq QD_Q,
\end{equation*}
and in particular $x\notin D_Q$. However, $V\subseteq D_Q$, and thus $x\notin V$, a contradiction. Hence, we must have $I\nsubseteq Q$.

In this case, there must be a prime ideal $P_1\in\Delta$ not containing $I$. Moreover, $I\cap P_1\nsubseteq Q$ too, and thus there is another prime $P_2\in\Delta$, $P_1\neq P_2$, not containing $I$. By Lemma \ref{lemma:spec-linord}, the prime ideals of $A$ inside $M$ are linearly ordered; in particular, we can suppose without loss of generality that $\rad(P_2A)\subseteq\rad(P_1A)$.

Let now $t\in P_2\setminus P_1$; then, $t\in\rad(P_1A)$, and thus there are $p_1,\ldots,p_k\in P_1$, $a_1,\ldots,a_n\in A$ such that $t^e=p_1a_1+\cdots+p_ka_k$ for some positive integer $e$. For each $i$, $a_i=B_i(x_1,\ldots,x_n)$, where $B_i$ is a polynomial over $D$ of total degree $d_i$; let $d:=\sup\{d_1,\ldots,d_k\}$, and take an $r\in I\setminus P_1$ (recall that $I\nsubseteq P_1$). Then, $r^dB_i(x_1,\ldots,x_n)\in D$ for each $i$; therefore,
\begin{equation*}
r^dt^e=p_1r^da_1+\cdots+p_kr^da_k\in p_1D+\cdots+p_kD\subseteq P_1.
\end{equation*}
However, by construction, both $r$ and $t$ are out of $P_1$; since $P_1$ is prime, this is impossible. Hence, $\Zar(D)\setminus\{V\}$ is not compact, and $\Zar(D)$ is not Noetherian.
\end{proof}

The first corollaries of this result can be obtained simply by putting $Q=(0)$. Recall that a \emph{G-domain} (or \emph{Goldman domain}) is an integral domain such that the intersection of all nonzero prime ideals is nonzero. They were introduced by Kaplansky for giving a new proof of Hilbert's Nullstellensatz (see for example \cite[Section 1.3]{kaplansky}).
\begin{cor}\label{cor:Goldman}
Let $D$ be a local domain of finite dimension, and suppose that $D$ is not a G-domain. Then, $\Zar(D)\setminus\{V\}$ is not compact for every $V\in\Zarmin(D)$.
\end{cor}
\begin{proof}
Since $D$ is finite-dimensional, every prime ideal of $D$ contains a prime ideal of height 1; since $D$ is not a G-domain, it follows that the intersection of the set $\Spec^1(D)$ of the height-1 prime ideals of $D$ is $(0)$. The localization $D_{(0)}$ is the quotient field of $D$, and thus a valuation domain; therefore, we can apply Theorem \ref{teor:intersez-P} to $\Delta:=\Spec^1(D)$.
\end{proof}

\begin{cor}\label{cor:h1}
Let $D$ be a local domain. If $D$ has infinitely many height-1 primes, then $\Zar(D)$ is not Noetherian.
\end{cor}
\begin{proof}
Let $I$ be the intersection of all height-1 prime ideals. If $I\neq(0)$, every height-one prime of $D$ would be minimal over $I$; since there is an infinite number of them, $\Spec(D)$ would not be Noetherian, and by Proposition \ref{prop:spec-noeth} neither $\Zar(D)$ would be Noetherian. Hence, $I=(0)$. But then we can apply Theorem \ref{teor:intersez-P} (for $Q=I$).
\end{proof}

Note that the hypothesis that $D$ is local is needed in Theorem \ref{teor:intersez-P} and in Corollary \ref{cor:h1}: for example, $\insZ$ has infinitely many height-1 primes, and $\bigcap\{P\mid P\in\Spec^1(D)\}=(0)$, but $\Zar(\insZ)\simeq\Spec(\insZ)$ is a Noetherian space.

\begin{prop}\label{prop:polynomial}
Let $D$ be an integral domain. If $D$ is not a field, then $\Zar(D[X])$ is a not a Noetherian space.
\end{prop}
\begin{proof}
Since $D$ is not a field, there exist a nonzero prime ideal $P$ of $D$. For any $a\in P$, let $\mathfrak{p}_a$ be the ideal of $D[X]$ generated by $X-a$; then, each $\mathfrak{p}_a$ is a prime ideal of height 1, $\mathfrak{p}_a\neq\mathfrak{p}_b$ if $a\neq b$, and $\bigcap\{\mathfrak{p}_a\mid a\in P\}=(0)$.

The prime ideal $\mathfrak{m}:=PD[X]+XD[X]$ contains every $\mathfrak{p}_a$; by Corollary \ref{cor:h1}, $\Zar(D[X]_\mathfrak{m})$ is not Noetherian. Therefore, neither $\Zar(D[X])$ is Noetherian.
\end{proof}

\begin{cor}\label{cor:FL}
Let $F\subseteq L$ be a transcendental field extension.
\begin{enumerate}[(a)]
\item\label{cor:FL:1} If $\trdeg_F(L)=1$ and $L$ is finitely generated over $F$ then $\Zar(L|F)$ is Noetherian.
\item\label{cor:FL:>1} If $\trdeg_F(L)>1$ then $\Zar(L|F)$ is not Noetherian.
\end{enumerate}
\end{cor}
\begin{proof}
\ref{cor:FL:1} Let $L=F(\alpha_1,\ldots,\alpha_n)$; without loss of generality we can suppose that $\alpha_1$ is transcendental over $F$. Then, the extension $F(\alpha_1)\subseteq L$ is algebraic and finitely generated, and thus finite.

Each $V\in\Zar(L|F)$ must contain either $\alpha_1$ or $\alpha_1^{-1}$; therefore, $\Zar(L|F)=\Zar(L|F[\alpha_1])\cup\Zar(L|F[\alpha_1^{-1}])$. However, $\Zar(L|A)=\Zar(A')$ for every domain $A$, where we denote by $A'$ is the integral closure of $A$ in $L$; since $F[\alpha_1]$ (respectively, $F[\alpha_1^{-1}]$) is a principal ideal domain and $F(\alpha_1)\subseteq L$ is finite, the integral closure of $F[\alpha_1]$ (resp., $F[\alpha^{-1}]$) is a Dedekind domain, and thus $\Zar(L|F[\alpha_1])=\Zar(F[\alpha_1]')\simeq\Spec(F[\alpha_1]')$ is Noetherian. Being the union of two Noetherian spaces, $\Zar(L|F)$ is itself Noetherian.

\ref{cor:FL:>1} Suppose $\trdeg_F(L)>1$. Then, there are $X,Y\in L$ such that $\{X,Y\}$ is an algebraically independent set over $F$; in particular, we have a continuous surjective map $\Zar(L|F)\longrightarrow\Zar(F(X,Y)|F)$ given by $V\mapsto V\cap F(X,Y)$. However, $\Zar(F(X,Y)|F)$ contains $\Zar(F[X,Y])$; by Proposition \ref{prop:polynomial}, the latter is not Noetherian, since $F[X,Y]$ is the polynomial ring over $F[X]$, a domain of dimension 1. Thus, $\Zar(L|F)$ is not Noetherian.
\end{proof}

The condition that $\bigcap\{P\mid P\in\Delta\}=Q$ of Theorem \ref{teor:intersez-P} can be slightly generalized, requiring only that the intersection is contained in $Q$. However, doing so we can only prove that $\Zar(D)$ is not Noetherian, without always finding a specific $V$ such that $\Zar(D)\setminus\{V\}$ is not compact.
\begin{prop}\label{prop:intersez-contenuto}
Let $D$ be a local integral domain, and suppose there is a set $\Delta\subseteq\Spec(D)$ and a prime ideal $Q$ such that:
\begin{enumerate}
\item $Q\notin\Delta$;
\item no two members of $\Delta$ are comparable;
\item $\bigcap\{P\mid P\in\Delta\}\subseteq Q$;
\item $D_Q$ is a valuation domain.
\end{enumerate}
Then, $\Zar(D)$ is not Noetherian.
\end{prop}
\begin{proof}
If $\Spec(D)$ is not Noetherian, by Proposition \ref{prop:spec-noeth} neither is $\Zar(D)$; suppose that $\Spec(D)$ is Noetherian.

Let $I:=\bigcap\{P\mid P\in\Delta\}$; since an overring of a valuation domain is still a valuation domain, we can suppose that $Q$ is a minimal prime of $I$. Since $D$ has Noetherian spectrum, the radical ideal $I$ has only a finite number of minimal primes, say $Q=:Q_1,Q_2,\ldots,Q_n$; let $\Delta_i:=\{\mathfrak{p}\in\Delta\mid Q_i\subseteq \mathfrak{p}\}$ and $I_i:=\bigcap\{\mathfrak{p}\mid \mathfrak{p}\in\Delta_i\}$. By standard properties of minimal primes, $\Delta=\Delta_1\cup\cdots\cup\Delta_n$ and $I=I_1\cap\cdots\cap I_n$.

In particular, $I_1\cap\cdots\cap I_n\subseteq Q$; hence, $I_k\subseteq Q$ for some $k$. However, $Q_k\subseteq I_k$, and thus $Q_k\subseteq Q$; since different minimal primes of the same ideal are not comparable, $k=1$ and $Q\subseteq I_1\subseteq Q$, i.e., $I_1=Q$. Then, $\Delta_1$ is a family of primes satisfying the hypothesis of Theorem \ref{teor:intersez-P}; in particular, $\Zar(D)$ is not Noetherian.
\end{proof}

An \emph{essential prime} of a domain $D$ is a $P\in\Spec(D)$ such that $D_P$ is a valuation domain. $D$ is an \emph{essential domain} if it is equal to the intersection of the localizations of $D$ at the essential primes. If, moreover, the family of the essential primes is compact, then $D$ can be called a \emph{Pr\"ufer $v$-multiplication domain} (\emph{P$v$MD} for short) \cite[Corollary 2.7]{fin-tar-PvMD-ess}; note that the original definition of P$v$MDs was given through star operations (more precisely, $D$ is a P$v$MD if and only if $D_P$ is a valuation ring for every $t$-maximal ideal $P$ \cite{griffin_vmultiplication_1967,kang_pvmd}).
\begin{prop}\label{prop:PvMD}
Let $D$ be an essential domain. Then, $\Zar(D)$ is Noetherian if and only if $D$ is a Pr\"ufer domain with Noetherian spectrum.
\end{prop}
\begin{proof}
If $D$ is a Pr\"ufer domain with Noetherian spectrum, then $\Zar(D)\simeq\Spec(D)$ is Noetherian (see Example \ref{ex:prufnoeth}). Conversely, suppose $\Zar(D)$ is Noetherian: by Proposition \ref{prop:spec-noeth}, $\Spec(D)$ is Noetherian. Let $\mathcal{E}$ be the set of essential prime ideals of $D$: since $\Spec(D)$ is Noetherian, $\mathcal{E}$ is compact, and thus $D$ is a P$v$MD.

Suppose by contradiction that $D$ is not a Pr\"ufer domain. Then, there is a maximal ideal $M$ of $D$ such that $D_M$ is not a valuation domain; since the localization of a P$v$MD is a P$v$MD \cite[Theorem 3.11]{kang_pvmd}, and $\Zar(D_M)$ is a subspace of $\Zar(D)$, without loss of generality we can suppose $D=D_M$, i.e., we can suppose that $D$ is local.

Since $\mathcal{E}$ is compact, every $P\in\mathcal{E}$ is contained in a maximal element of $\mathcal{E}$; let $\Delta$ be the set of such maximal elements. Clearly, $D=\bigcap\{D_P\mid P\in\Delta\}$. If $\Delta$ were finite, $D$ would be an intersection of finitely many valuation domains, and thus it would be a Pr\"ufer domain \cite[Theorem 22.8]{gilmer}; hence, we can suppose that $\Delta$ is infinite. Let $I:=\bigcap\{P\mid P\in\Delta\}$.

Each $P\in\Delta$ contains a minimal prime of $I$; however, since $\Spec(D)$ is Noetherian, $I$ has only finitely many minimal primes. It follows that there is a minimal prime $Q$ of $I$ that is not contained in $\Delta$; in particular, $\bigcap\{P\mid P\in\Delta\}\subseteq Q$, and thus we can apply Proposition \ref{prop:intersez-contenuto}. Hence, $\Zar(D)$ is not Noetherian, which is a contradiction.
\end{proof}

\begin{oss}
The previous proof can be interpreted using the terminology of the theory of star operations. Indeed, any essential prime $P$ is a $t$-ideal, i.e., $P=P^t$, where (for any ideal $J$ of $D$) $J^t:=\bigcup\{(D:(D:I))\mid I\subseteq J\text{~is finitely generated}\}$ \cite[Lemma 3.17]{kang_pvmd} and if $D$ is a P$v$MD then the set $\Delta$ of the maximal elements of $\mathcal{E}$ is exactly the set of \emph{$t$-maximal ideals}, i.e., the set of the ideals $I$ such that $I=I^t$ and $J\neq J^t$ for every proper ideal $I\subsetneq J$.
\end{oss}

\begin{cor}
Let $D$ be a Krull domain. Then, $\Zar(D)$ is Noetherian if and only if $\dim(D)=1$, i.e., if and only if $D$ is a Dedekind domain.
\end{cor}
\begin{proof}
If $\dim(D)=1$ then $D$ is Noetherian and so is $\Zar(D)$. If $\dim(D)>1$, then $D$ is not a Pr\"ufer domain; since each Krull domain is a P$v$MD, we can apply Proposition \ref{prop:PvMD}.
\end{proof}
Note that this corollary can also be proved directly from Corollary \ref{cor:h1} since, if $D$ is Krull, and $P\in\Spec(D)$ has height 2 or more, then $D_P$ has infinitely many height-1 primes.

\section{An application: Kronecker function rings}\label{sect:kronecker}
Let $D$ be an integrally closed integral domain with quotient field $K$. For every $V\in\Zar(D)$, let $V(X):=V[X]_{\mathfrak{m}_V[X]}\subseteq K(X)$, where $\mathfrak{m}_V$ is the maximal ideal of $V$. If $\Delta\subseteq\Zar(D)$, the \emph{Kronecker function ring} of $D$ with respect to $\Delta$ is
\begin{equation*}
\Kr(D,\Delta):=\bigcap\{V(X)\mid V\in\Delta\};
\end{equation*}
equivalently,
\begin{equation*}
\Kr(D,\Delta)=\{f/g\mid f,g\in D[X], g\neq 0, \cont(f)\subseteq(\cont(g))^{\wedge_\Delta}\},
\end{equation*}
where $\cont(f)$ is the content of $f$ and $\wedge_\Delta$ is the semistar operation defined in Section \ref{sect:semistar}. See \cite[Chapter 32]{gilmer} or \cite{fontana_loper} for general properties of Kronecker function rings.

The set of Kronecker function rings it exactly the set of overrings of the basic Kronecker function ring $\Kr(D,\Zar(D))$; this set is in bijective correspondence with the set of finite-type valuative semistar operations \cite[Remark 32.9]{gilmer}, or equivalently with the set of nonempty subsets of $\Zar(D)$ that are closed in the inverse topology \cite[Theorem 4.9]{fifolo_transactions}.

Let $\Krset(D)$ be the set of Kronecker function rings $T$ of $D$ such that $T\cap K=D$. Then, $\Krset(D)$ is in bijective correspondence with the set of finite-type valuative \emph{star} operations, or equivalently with the set of inverse-closed representation of $D$ through valuation rings, i.e., the sets $\Delta\subseteq\Zar(D)$ that are closed in the inverse topology and such that $\bigcap\{V\mid V\in\Delta\}=D$ \cite[Proposition 5.10]{olberding_affineschemes}.

It has been conjectured \cite{mcg-kronecker-pers} that $\Krset(D)$ is either a singleton (in which case $D$ is said to be a \emph{vacant domain}; see \cite{vacantdomains}) or infinite, and this has been proved to be the case for a wide class of pseudo-valuation domains \cite[Theorem 4.10]{vacantdomains}. As a consequence of the following proposition, we will prove this conjecture for another class of domains.

\begin{prop}\label{prop:Krset-comp}
Let $D$ be an integrally closed integral domain such that $1<|\Krset(D)|<\infty$. Then, there is a minimal valuation overring $V$ of $D$ such that $\Zar(D)\setminus\{V\}$ is compact.
\end{prop}
\begin{proof}
Suppose $|\Krset(D)|>1$. Then, there is an inverse-closed representation $\Delta$ of $D$ different from $\Zar(D)$; let $\Lambda:=\Zar(D)\setminus\Delta$. For each $W\in\Lambda$, let $\Delta(W):=\Delta\cup\{W\}^\uparrow$; then, every $\Delta(W)$ is an inverse-closed representation of $D$, and $\Delta(W)\neq\Delta(W')$ if $W\neq W'$ (since, without loss of generality, $W\nsupseteq W'$, and thus $W\notin\Delta(W')$). Hence, each $W\in\Lambda$ give rise to a different member of $\Krset(D)$; since $|\Krset(D)|<\infty$, it follows that $\Lambda$ is finite.

If now $V$ is minimal in $\Lambda$, then $\Zar(D)\setminus\{V\}=\Delta\cup(\Lambda\setminus\{V\})$ is closed by generizations; since $\Lambda$ is finite, it follows that $\Zar(D)\setminus\{V\}$ is the union of two compact subspaces, and thus it is itself compact.
\end{proof}

\begin{cor}
Let $D$ be an integrally closed local integral domain, and suppose there exist a set $\Delta\subseteq\Spec(D)$ of incomparable nonzero prime ideals such that $\bigcap\{P\mid P\in\Delta\}=(0)$. Then, $|\Krset(D)|\in\{1,\infty\}$.
\end{cor}
\begin{proof}
By Theorem \ref{teor:intersez-P}, each $\Zar(D)\setminus\{V\}$ is noncompact. The claim now follows from Proposition \ref{prop:Krset-comp}.
\end{proof}

\section{Overrings of Noetherian domains}\label{sect:Noeth}
If $D$ is a Noetherian domain, Theorem \ref{teor:Zar-meno-V} admits a direct application, without using any of the results proved in Sections \ref{sect:dimV} and \ref{sect:intersection}. Indeed, if $D$ is Noetherian with quotient field $K$, then it is the same for any localization of $D[x_1,\ldots,x_n]$, for arbitrary $x_1,\ldots,x_n\in K$; thus, the integral closure of $D[x_1,\ldots,x_n]_M$ is a Krull domain for each maximal ideal $M$ of $D[x_1,\ldots,x_n]$ (\cite[(33.10)]{nagata_localrings} or \cite[Theorem 4.10.5]{swanson_huneke}). Since a domain that is both Krull and a valuation ring must be a field or a discrete valuation ring, Theorem \ref{teor:Zar-meno-V} implies that $\Zar(D)\setminus\{V\}$ is not compact as soon as $V$ is a minimal valuation overring of dimension 2 or more.

We can actually say more than this; the following is a proof through Proposition \ref{prop:comp-Zarmin} of an observation already appeared in \cite[Example 3.7]{surveygraz}.
\begin{prop}\label{prop:noethdom}
Let $D$ be a Noetherian domain with quotient field $K$, and let $\Delta$ be the set of valuation overrings of $D$ that are Noetherian (i.e., $\Delta$ is the union of $\{K\}$ with the set of discrete valuation overrings of $D$). Then, $\Delta$ is compact if and only if $\dim(D)=1$.
\end{prop}
\begin{proof}
If $\dim(D)=1$, then $\Delta=\Zar(D)$, and thus it is compact.

On the other hand, for every ideal $I$ of $D$, $I^{\wedge_\Delta}=I^b$ \cite[Proposition 6.8.4]{swanson_huneke}; however, if $\dim(D)>1$, then $\Zar(D)$ contains elements of dimension 2, and thus $\Delta$ cannot contain $\Zarmin(D)$. The claim now follows from Proposition \ref{prop:comp-Zarmin}.
\end{proof}

\begin{oss}
~\\
\begin{enumerate}
\item The equality $I^{\wedge_\Delta}=I^b$ holds also if we restrict $\Delta$ to be the set of discrete valuation overrings of $D$ whose center is a maximal ideal of $D$ \cite[Proposition 6.8.4]{swanson_huneke}. For each prime ideal of height 2 or more, by passing to $D_P$, we can thus prove that the set of discrete valuation overrings of $D$ with center $P$ is not compact (and in particular it is infinite).
\item The previous proposition also allows a proof of the second part of Corollary \ref{cor:FL} without using Theorem \ref{teor:intersez-P}, since $F[X,Y]$ is a Noetherian domain of dimension 2.
\end{enumerate}
\end{oss}

By Proposition \ref{prop:noethdom}, in particular, the space $\Delta$ of Noetherian valuation overrings of $D$ (where $D$ is Noetherian and $\dim(D)\geq 2$) is not a spectral space, since it is not compact. Our next purpose is to see $\Delta$ as an intersection $X\cap\Zar(D)$, for some subset $X$ of $\Over(D)$, and use this representation to prove facts about $X$. We start with using the inverse topology.
\begin{prop}\label{prop:noeth-chiusinv}
Let $D$ be a Noetherian domain with quotient field $K$, and let:
\begin{itemize}
\item $X_1$ be the set of all overrings of $D$ that are Noetherian and of dimension at most 1;
\item $X_2$ be the set of all overrings of $D$ that are Dedekind domains ($K$ included).
\end{itemize}
For $i\in\{1,2\}$, the following are equivalent:
\begin{enumerate}[(i)]
\item\label{prop:noeth-chiusinv:comp} $X_i$ is compact;
\item\label{prop:noeth-chiusinv:sp} $X_i$ is spectral;
\item\label{prop:noeth-chiusinv:cons} $X_i$ is proconstructible in $\Over(D)$; 
\item\label{prop:noeth-chiusinv:dim} $\dim(D)=1$.
\end{enumerate}
\end{prop}
\begin{proof}
\ref{prop:noeth-chiusinv:comp} $\Longrightarrow$ \ref{prop:noeth-chiusinv:cons}. In both cases, $X=X^{\gen}$: for $X_1$ see \cite[Theorem 93]{kaplansky}, while for $X_2$ see e.g. \cite[Theorem 40.1]{gilmer} (or use the previous result and \cite[Corollary 36.3]{gilmer}). \ref{prop:noeth-chiusinv:cons} $\Longrightarrow$ \ref{prop:noeth-chiusinv:sp} $\Longrightarrow$ \ref{prop:noeth-chiusinv:comp} always holds.

\ref{prop:noeth-chiusinv:dim} $\Longrightarrow$ \ref{prop:noeth-chiusinv:comp}. If $\dim(D)=1$, then $X_1=\Over(D)$, while $X_2=\Over(D')$, where $D'$ is the integral closure of $D$, and both are compact since they have a minimum.

\ref{prop:noeth-chiusinv:cons} $\Longrightarrow$ \ref{prop:noeth-chiusinv:dim}. If $X_i$ is proconstructible, so is $X_i\cap\Zar(D)$ (since $\Zar(D)$ is also proconstructible), and in particular $X_i\cap\Zar(D)$ is compact. However, in both cases, $X_i\cap\Zar(D)$ is exactly the set of Noetherian valuation overrings of $D$; by Proposition \ref{prop:noethdom}, $\dim(D)=1$.
\end{proof}

\begin{oss}
The equivalence between the first three conditions of Proposition \ref{prop:noeth-chiusinv} holds for every subset $X\subseteq\Over(D)$ such that $X=X^\gen$ (and every domain $D$). In particular, it holds if $X$ is the set of overrings of $D$ that are principal ideal domains, and, with the same proof of the other cases, we can show that if $D$ is Noetherian and these conditions hold, then $\dim(D)=1$. However, it is not clear if, when $D$ is Noetherian and $\dim(D)=1$, this set is actually compact.
\end{oss}

Another immediate consequence of Proposition \ref{prop:noethdom} is that the set $\NoethOver(D)$ of Noetherian overrings of $D$ is not proconstructible as soon as $D$ is Noetherian and $\dim(D)\geq 2$: indeed, if it were, then $\NoethOver(D)\cap\Zar(D)=\Delta$ would be proconstructible, against the fact that $\Delta$ is not compact. However, this is also a consequence of a more general result. We need a topological lemma.
\begin{lemma}\label{lemma:YX-cons}
Let $Y\subseteq X$ be spectral spaces. Suppose that there is a subbasis $\mathcal{B}$ of $X$ such that, for every $B\in\mathcal{B}$, both $B$ and $B\cap Y$ are compact. Then, $Y$ is a proconstructible subset of $X$.
\end{lemma}
\begin{proof}
The hypothesis on $\mathcal{B}$ implies that the inclusion map $Y\hookrightarrow X$ is a spectral map; by \cite[1.9.5(vii)]{EGA4-1}, it follows that $Y$ is a proconstructible subset of $X$.
\end{proof}

\begin{prop}\label{prop:denso-costr}
Let $D$ be an integral domain with quotient field $K$, and let $D[\insfracidfg]$ be the set of finitely generated $D$-algebras contained in $K$.
\begin{enumerate}[(a)]
\item\label{prop:denso-costr:alg} $D[\insfracidfg]$ is dense in $\Over(D)$, with respect to the inverse topology.
\item\label{prop:denso-costr:sp} Let $X$ such that $D[\insfracidfg]\subseteq X\subseteq\Over(D)$. Then, $X$ is spectral in the Zariski topology if and only if $X=\Over(D)$.
\end{enumerate} 
\end{prop}
\begin{proof}
\ref{prop:denso-costr:alg} A basis of the constructible topology is given by the sets of type $U\cap(X\setminus V)$, as $U$ and $V$ ranges in the open and compact subsets of $\Over(D)$. Such an $U$ can be written as $B_1\cup\cdots\cup B_n$, where each $B_i=B(x_1^{(i)},\ldots,x_n^{(i)})$ is a basic open set of $\Over(D)$; thus, we can suppose that $U=B(x_1,\ldots,x_n)$. Suppose $\Omega:=U\cap(X\setminus V)$ is nonempty; we claim that $A:=D[x_1,\ldots,x_n]\in\Omega\cap D[\insfracidfg]$. Clearly $A\in D[\insfracidfg]$ and $A\in U$; let $T\in\Omega$. Then, $T\in U$, and thus $A\subseteq T$; therefore, $A$ is in the closure $\Cl(T)$ of $T$, with respect to the Zariski topology. But $X\setminus V$ is closed, and thus $\Cl(T)\subseteq X\setminus V$; i.e., $A\in X\setminus V$. Hence, $A\in\Omega\cap D[\insfracidfg]$, which in particular is nonempty, and $D[\insfracidfg]$ is dense.

\ref{prop:denso-costr:sp} Suppose $X$ is spectral. For every $x_1,\ldots,x_n$, the set $X\cap B(x_1,\ldots,x_n)$ has a minimum (i.e., $D[x_1\ldots,x_n]$), so it is compact. Since the family of all $B(x_1,\ldots,x_n)$ is a basis, by Lemma \ref{lemma:YX-cons} it follows that $X$ is proconstructible. By the previous point, we must have $X=\Over(D)$.
\end{proof}

\begin{cor}
Let $D$ be a Noetherian domain. The spaces
\begin{itemize}
\item $\NoethOver(D):=\{T\in\Over(D)\mid T$ is Noetherian$\}$, and
\item $\KrullOver(D):=\{T\in\Over(D)\mid T$ is a Krull domain$\}$
\end{itemize}
are spectral if and only if $\dim(D)=1$.
\end{cor}
\begin{proof}
If $\dim(D)=1$, then the claim follows by Proposition \ref{prop:noeth-chiusinv}.

If $\dim(D)\geq 2$, then $\NoethOver(D)$ is not spectral by Proposition \ref{prop:denso-costr}\ref{prop:denso-costr:sp} and the Hilbert Basis Theorem; the case of $\KrullOver(D)$ follows in the same way, since $\KrullOver(D)\cap B(x_1,\ldots,x_n)$ has always a minimum (i.e., the integral closure of $D[x_1,\ldots,x_n]$).
\end{proof}

More generally, consider a property $\mathcal{P}$ of Noetherian domains such that every field and every discrete valuation ring satisfies $\mathcal{P}$; for example, $\mathcal{P}$ may be the property of being regular, Gorenstein or Cohen-Macaulay. Let $X_\mathcal{P}(D)$ be the set of overrings of $D$ satisying $\mathcal{P}$; then, $X_\mathcal{P}(D)\cap\Zar(D)$ is not compact, and thus $X_\mathcal{P}(D)$ is not proconstructible. On the other hand, if $X_\mathcal{P}(T)$ is compact for every overring of $D$ that is finitely generated as a $D$-algebra, then by Lemma \ref{lemma:YX-cons} it follows that $X_\mathcal{P}(D)$ cannot be a spectral space. Thus, the assignment $D\mapsto X_\mathcal{P}(D)$ cannot be ``too good'': either some $X_\mathcal{P}(T)$ is not compact, or $X_\mathcal{P}(D)$ is not spectral.

\vspace{5mm}

\textbf{Question.} Let $\mathcal{P}$ be the property of being regular, the property of being Gorenstein or the property of being Cohen-Macaulay. Is it possible to characterize for which Noetherian domains $D$ there is a $T\in\Over(D)$ such that $X_\mathcal{P}(T)$ is not compact and for which $X_\mathcal{P}(D)$ is not spectral?

\section{Acknowledgments}
Section \ref{sect:kronecker} was inspired by a talk given by Daniel McGregor at the congress ``Recent Advances in Commutative ring and Module Theory'' in Bressanone (June 13-17, 2017). I also thank the referee for his numerous suggestions, which improved the paper.

\end{document}